\documentclass[preprint,12pt,authoryear]{elsarticle}
\usepackage{graphicx}
\usepackage{subcaption}
\usepackage{amssymb}
\usepackage{amsmath}
\usepackage{float}
\usepackage{siunitx}
\DeclareSIUnit{\mile}{mi}
\journal{Transportation Research Part B}

\usepackage{amsthm}
\usepackage{mathtools}

\theoremstyle{plain} % Style for theorems, propositions, lemmas, corollaries
\newtheorem{theorem}{Theorem}[section]
\newtheorem{proposition}[theorem]{Proposition}

\newtheorem{corollary}[theorem]{Corollary}

\theoremstyle{definition} % Style for definitions, examples

\theoremstyle{remark} % Style for remarks, notes

\begin{document}

\begin{frontmatter}

\title{Improving Travel Time Reliability with Variable Speed Limits}

\author[1]{Alexander Hammerl}
\author[1]{Ravi Seshadri}
\author[1]{Thomas Kjær Rasmussen}
\author[1]{Otto Anker Nielsen}

\affiliation[1]{organization={Technical University of Denmark, Department of Technology, Management and Economic},%Department and Organization
            addressline={Akademivej 358}, 
            city={Kgs. Lyngby},
            postcode={2800}, 
            state={},
            country={Denmark}}

%% Abstract
\begin{abstract}
This paper analyzes the use of variable speed limits to optimize travel time reliability for commuters. The investigation focuses on a traffic corridor with a bottleneck subject to the capacity drop phenomenon. The optimization criterion is a linear combination of the expected value and standard deviation of average travel time, with traffic flow dynamics following the kinematic wave model (Lighthill, 1955; Richards, 1956). We develop two complementary models to optimally set variable speed limits: In the first model, daily peak traffic demand is conceptualized as a stochastic variable, and the resulting model is solved through a three-stage optimization algorithm. The second model is based on deterministic demand, instead modeling bottleneck capacity as a stochastic process using a stochastic differential equation (SDE). The practical applicability of both approaches is demonstrated through numerical examples with empirically calibrated data.
\end{abstract}


\begin{highlights}
\item Optimization of variable speed limits to balance expected travel times and their deviations
\item Threshold-based control strategy for stochastic peak demand using variational methods
\item Stochastic MPC formulation with capped Ornstein–Uhlenbeck model for capacity fluctuations
\item Simulation-based validation on I-880N showing significant reduction in travel time variability
\end{highlights}

%% Keywords

\end{frontmatter}
\section{Introduction}
Variable Speed Limit (VSL) systems are dynamic traffic management tools that adjust posted speed limits in response to real-time traffic conditions through overhead or roadside display signs. A review paper by \cite{Khondaker2015} summarizes the operational benefits of traditional VSL applications as follows: improved safety, prevention of traffic breakdown, and increase of throughput at bottlenecks. The latter two are possible due to the capacity drop phenomenon, where the maximum flow rate at an active bottleneck can drop below the theoretical capacity once congestion sets in (\cite{Banks1990}, \cite{Hall1991}, \cite{Cassidy1999}, \cite{Bertini2005}) - a reduction that VSL systems can help prevent by strategically slowing down approaching traffic. Modern VSL control strategies are based on predictive modeling of traffic behavior (e.g. \cite{Bertini2006},\cite{Papageorgiou2008}, \cite{Carlson2010}, \cite{Carlson2010a}, \cite{Carlson2010b}, \cite{Carlson2011}, \cite{Weikl2013}, \cite{Chen2014}, \cite{Carlson2014}, \cite{Jin2015}, \cite{Khondaker2015a}, \cite{Zhang2018}). They employ a closed-loop feedback control mechanism to continuously update predictions and control actions based on new traffic data. These predictions typically rely on either macroscopic traffic flow models (c.f. \cite{Lighthill1955},\cite{Richards56},\cite{Payne1971}) or microscopic models of car-following behavior (c.f. \cite{Bando1995},\cite{Treiber2000},\cite{Helbing2001}).

Much less attention has been paid to the potential of VSL systems for improving travel time reliability. Travel time reliability, typically measured by its standard deviation or variance, represents a significant factor in commuter route choice. While some empirical studies suggesting travelers value predictability nearly as highly as average travel time, all studies confirm its substantial importance \citep{Li2010, Ramezani2012, Prato2014, seshadri2017robust, prakash2018consistent, Saedi2020}. To our knowledge, the optimization of VSL systems to improve travel time reliability has not been investigated in existing literature. This paper addresses this research gap and develops methods for specifically optimizing VSL strategies for travel time reliability.

Our primary contributions are as follows: first, we formulate a novel bi-objective optimization framework that explicitly balances mean travel time against travel time variability, extending traditional VSL approaches that focus solely on throughput maximization or mean travel time minimization. This framework captures commuters' risk-averse route choice behavior directly in the control objective. Second, we develop two complementary stochastic models to capture different sources of uncertainty in traffic systems: a discrete stochastic programming formulation for day-to-day demand variations and a continuous-time stochastic differential equation model for within-day capacity fluctuations. Both models lead to computationally tractable solution algorithms - the demand uncertainty model reduces to a one-dimensional threshold search problem despite its infinite-dimensional nature, while the capacity uncertainty model yields a convex stochastic model predictive control formulation solvable in real-time. Additionally, we provide theoretical insights into the structure of optimal policies. We prove that under demand uncertainty, the optimal VSL strategy exhibits an elegant threshold structure that partitions traffic conditions into distinct control regimes, which significantly simplifies implementation.
The remainder of this paper is organized as follows. Section 2 presents the kinematic wave model with capacity drop that forms the foundation of our analysis. Section 3 develops the stochastic demand model, deriving the threshold-based optimal control policy through variational methods and presenting the Model Predictive Control implementation. Section 4 formulates the alternative model with stochastic capacity variations, employing an Ornstein-Uhlenbeck process to capture bottleneck capacity fluctuations and deriving the corresponding stochastic MPC formulation. Both sections include numerical validation using empirically calibrated parameters from Interstate 880 in California. Finally, Section 5 concludes with a discussion of practical implications and directions for future research.
\section{A Kinematic Wave Model of Capacity Drop}

The contemporary formulation of the LWR theory can be summarized as follows (see e.g. \cite{Jin2021}): The rate of change in the total number of vehicles contained in any road segment \( [x_1, x_2] \) where $x_2>x_1$ is equal to the net flow of vehicles out of the segment, i.e.
\begin{equation}
\frac{d}{dt} \int_{x_1}^{x_2} k(x,t) \, dx = - \left[ q(x,t) \right]_{x_1}^{x_2},
\label{eq:integralconservation}
\end{equation}
If $k$ and $q$ are differentiable functions, the expression simplifies to the partial differential equation
\begin{equation}
	\frac{\partial k}{\partial t} + \frac{\partial q}{\partial x} = 0.
	\label{eq:conservation}
\end{equation}
In addition to \ref{eq:integralconservation}, the LWR theory assumes the existence of a functional relationship between $q$ and $k$ under differentiable conditions:
\begin{equation}
	q(x,t)=Q(x,k(x,t)),
	\label{eq:fundamental}
\end{equation}
where $Q$ is a concave, non-negative function that is equal to zero at $k=0$ and at the \textit{jam density} $k=k_j$.
Flow and density are related to the cumulative flow \(N(x,t)\) as follows:

\begin{equation}
	q(x,t) = \frac{\partial N}{\partial t} (x,t), \quad k(x,t) = -\frac{\partial N}{\partial x} (x,t)
\end{equation}

In cases where \(k\) has a discontinuity at \((x,t)\), known as a shockwave, the shockwave's speed \(u\) is specified as:
\begin{equation}
	u = \frac{[q]}{[k]} = \frac{q_2 - q_1}{k_2 - k_1}.
	\label{eq:rhjump}
\end{equation}

The third variable, the \textit{average speed}, is defined as $v=\frac{q}{k}$. On substituting equation \ref{eq:conservation} into \ref{eq:fundamental}, we obtain 
\begin{equation}
	\frac{\partial k}{\partial t} + \frac{dQ}{dk} \cdot \frac{\partial k}{\partial x} = 0,
	\label{eq:differentiallwr}
\end{equation}It is important to note that customary solutions of the LWR model do \textbf{not} solve the coupled system of equations \ref{eq:integralconservation} and \ref{eq:fundamental}. Instead, they solve equation \ref{eq:differentiallwr}, which applies the fundamental diagram only at points where traffic variables are differentiable in both space and time. At discontinuities, solutions require supplementary entropy conditions to guarantee uniqueness. For comprehensive treatments of entropy solutions in traffic flow modeling, readers may consult \cite{Lebacque1996}, \cite{Ansorge1990} and \cite{Jin2009}. We assume that all considered instances are well-posed in the sense of (\cite{Daganzo2005a}, \cite{Daganzo2005b}, \cite{Daganzo2006}), i.e., every point in the solution domain is intersected by at least one kinematic wave. Under this assumption, the physically correct solution for $N(x,t)$ is uniquely determined by identification of the path from the boundary to $(x,t)$ along which the minimum solution for $N(x,t)$ is obtained (\cite{Daganzo2005a}, \cite{Newell1993}). The formulation of additional entropy solutions is not necessary. If a solution exists for $k(x,t)$ and $q(x,t)$, it is also unique. We proceed by defining the demand and supply functions by
\begin{align*}
d(k) &=  q\left(\min\{k, k_c\}\right), \\
s(k) &= q\left(\max\{k, k_c\}\right).
\end{align*}
Traffic moves along a road segment of length $l$, which ends in a bottleneck with a maximum capacity $q_{bn}$. When congestion forms at the bottleneck, its discharge capacity decreases by $\Delta$ percent. We model capacity drop according to the phenomenological approach by \cite{Jin2015cap}:

\begin{equation}
\label{eq:phen_cd}
q(l, t) =
\begin{cases}
d(l^-, t), & d(l^-, t) \leq s(0^+, t) \\
\min\{s(l^+, t), q_{bn}^*\}, & d(l^-, t) > s(l^+, t)
\end{cases}
\end{equation}

where $q_{bn}^*=(1-\Delta) \cdot q_{bn}$ is the dropped capacity. 
Theoretical explanations for the capacity drop phenomenon (see e.g.\cite{Hall1991},\cite{Jin2017},\cite{Jin2018},\cite{Khoshyaran2015},\cite{Wada2020}) primarily attribute the problem to drivers' bounded acceleration when exiting congested zones. \cite{Jin2017} demonstrates that the capacity drop model governed by \ref{eq:phen_cd} is formally equivalent to a behavioral kinematic wave representation of capacity drop with bounded acceleration inside a continuous lane drop bottleneck, except for the solution in a short acceleration zone immediately downstream of the bottleneck. 

The travel time $\tau(t)$ for a vehicle entering the segment at time $t$ is described by:
\begin{equation}
	\tau(t) = \inf \{ T \geq 0 : N(l, t + T) > N(0, t) \}.
\end{equation}

%To model rush hour traffic, the upstream boundary flow $q(0,t)$ is represented as a triangular function with a randomly distributed peak $q_p \sim \varphi$, defined as:

%\begin{equation}
%	q(0,t) = 
%	\begin{cases} 
%		q_b + a \cdot t, & \text{for } 0 \leq t \leq \frac{q_p-q_b}{a}, \\
%		q_p - b \cdot (t_{e} - t), & \text{for } \frac{q_p-q_b}{a} \leq t \leq \frac{q_p}{q_e \cdot b}+\frac{q_p-q_b}{a}, \\
%		q_e, & \text{for } \frac{q_p}{q_e \cdot b}+\frac{q_p-q_b}{a} \leq t \leq \infty.
%	\end{cases}
%\end{equation}

%for suitably chosen parameters \( q_b \) (initial flow), \( q_e \) (end flow), \( a \) (flow increase rate at the onset of congestion), and \( b \) (flow reduction rate at the offset of congestion). We further assume that $q_e<q_{bn}$, so that the expected travel time for very late departure times, as $t_{\text{dep}} \to \infty$, approaches the free flow travel time $\tau_{\text{free}}=\frac{l}{v(0)}$. 

\section{Variable Speed Limits and Stochastic Demand}
In this section, we determine variable speed limits via a time-stepped optimization that, 
at each decision time $t_k$, uses only the information available up to $t_k$. We consider unidirectional traffic flow along a homogeneous corridor without entry or exit ramps. Figure \ref{figure:schematic} depicts the schematic geometry of the analyzed road section. 

\begin{figure}[H]
    \centering
    \includegraphics[width=0.8\textwidth]{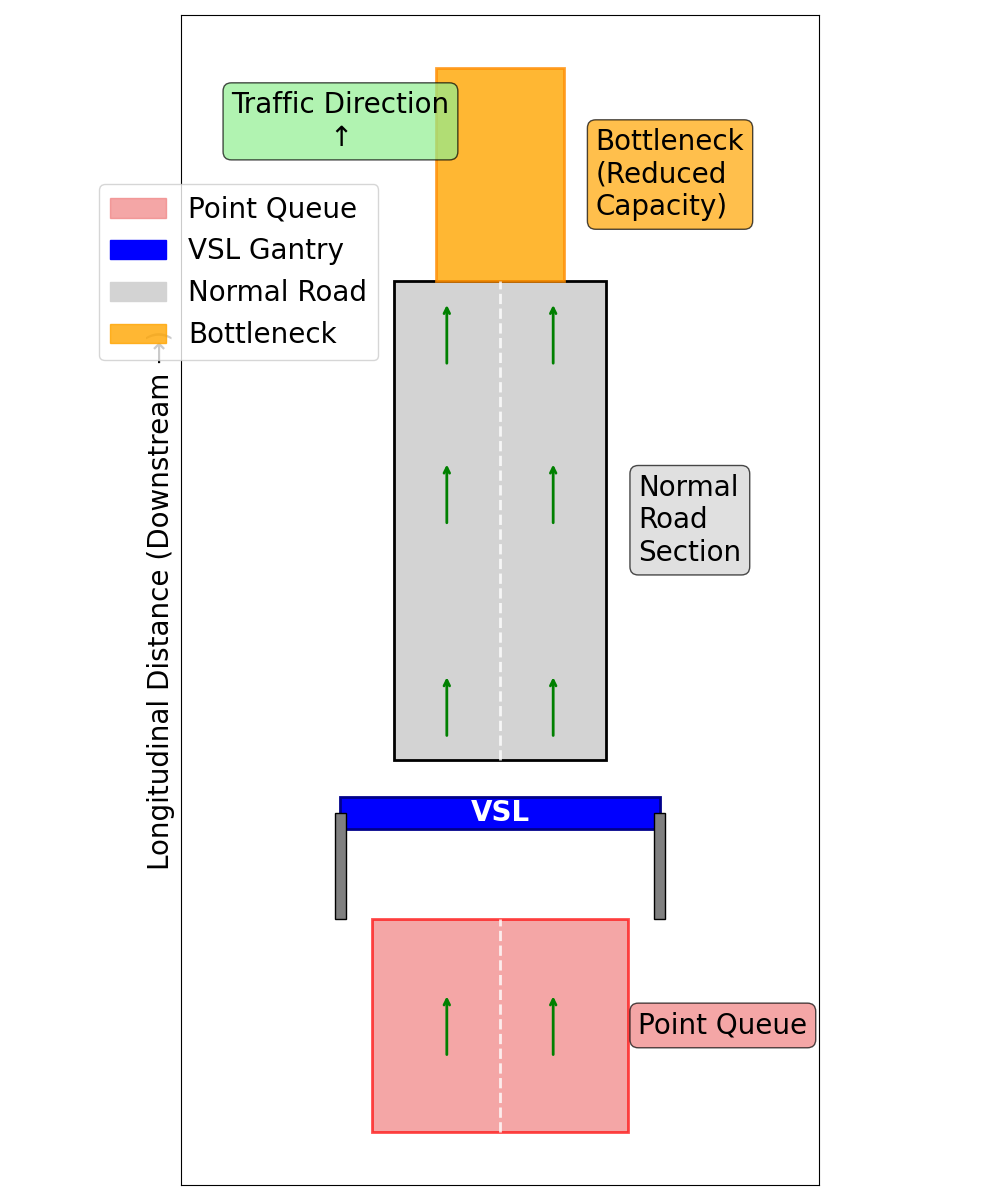}
    \caption{Schematic representation of the analyzed corridor with VSL control at the upstream boundary and bottleneck at the downstream end.}
    \label{figure:schematic}
\end{figure}

The objective balances the day-to-day mean and variability of the per-vehicle average travel time:
\begin{equation}
\label{eq:optimization_functional}
\min J := \alpha\,\mathbb{E}[\mathcal{T}_{\text{avg}}] + (1-\alpha)\,\text{Std}[\mathcal{T}_{\text{avg}}],\qquad \alpha\in[0,1],
\end{equation}
where the average travel time per vehicle is defined as
\[
\mathcal{T}_{\text{avg}} = \frac{\mathcal{T}}{\mathcal{Q}},
\]
with the total daily travel time and total admitted flow given by
\[
\mathcal{T} = \int_0^T \int_0^{\ell} k(x,t;q_p^0)\,dx\,dt, 
\qquad
\mathcal{Q} = \int_0^T q(0,t;q_p^0)\,dt.
\]
Here, $q(0,t;q_p^0)$ denotes the admitted inflow at the upstream boundary, $k(x,t;\allowbreak q_p^0)$
the traffic density, $\ell$ the corridor length, $T$ the optimization horizon, and $q_p^0$ represents the stochastic parameter (e.g., demand uncertainty) from which the randomness in $\mathcal{T}_{\text{avg}}$ arises.

\subsection{Computing the Minimum Total Travel Time}
In the first optimization step, the minimum achievable average travel time is determined for a given realization of the stochastic parameter $q_p^0$. For each fixed value of $q_p^0$, the boundary inflow profile $q(0,t;q_p^0)$ is deterministic and can take an arbitrary functional form. The control strategy is calibrated to minimize the average travel time for this specific realization. Since the total admitted flow is fixed by the inflow profile, minimizing the average travel time is equivalent to minimizing the total travel time, which in turn corresponds to maximizing the downstream flow $q(\ell,t)$ at every time $t$.

We define $N^+(l,t)$ as the number of vehicles that would pass position $l$ by time $t$ without a bottleneck, and $q^+(l,t)$ as the corresponding flow. According to the variational formulation of kinematic wave theory developed by \cite{Daganzo2005a}, the curve \( N^+(l, t) \) can be determined by
\[
N^+(l, t) = \min_{t_0 \leq t} \left\{ N(0, t_0) + f^*(\frac{l}{t_0 -t}) \cdot \frac{l}{t - t_0} \right\}, 
\]
where $f^*(p)$ denotes the Legendre-Fenchel transform of the fundamental diagram, 
\[
f^*(p)=\sup_k \{ q(k) -k \cdot p \}. 
\]
For the important class of concave speed-density relationships, \cite{Hammerl2024} note that the correct solution always corresponds to the kinematic wave that emanates last from the upstream end of the corridor.

The actual vehicle count $N_{\text{min}}(l,t)$ is determined using $q^+(l,t)$ as arrival rate and $q(l,t)$ as departure rate in a $D/D/1$ queue with service rate $q_{\text{bn}}$ (c.f. \cite{Newell1993}), as illustrated in Figure \ref{fig:queuing}. The minimum total travel time is then:
\[
\mathcal{T}_{\text{min}}(q_p^0) = \int_0^N \left(N^{-1}_{\text{min}}(l,t) - N^{-1}_{\text{min}}(0,t)\right) \, dt,
\]

the minimum average travel time is then given by 
\begin{equation}
\label{eq:minavg}
    \mathcal{T}_{\text{min}, \text{avg}}(q_p^0) = \frac{\mathcal{T}_{\text{min}}(q_p^0)}{N(0,T)}.
\end{equation}

\begin{figure}[ht]
    \centering
    \includegraphics[width=0.5\textwidth]{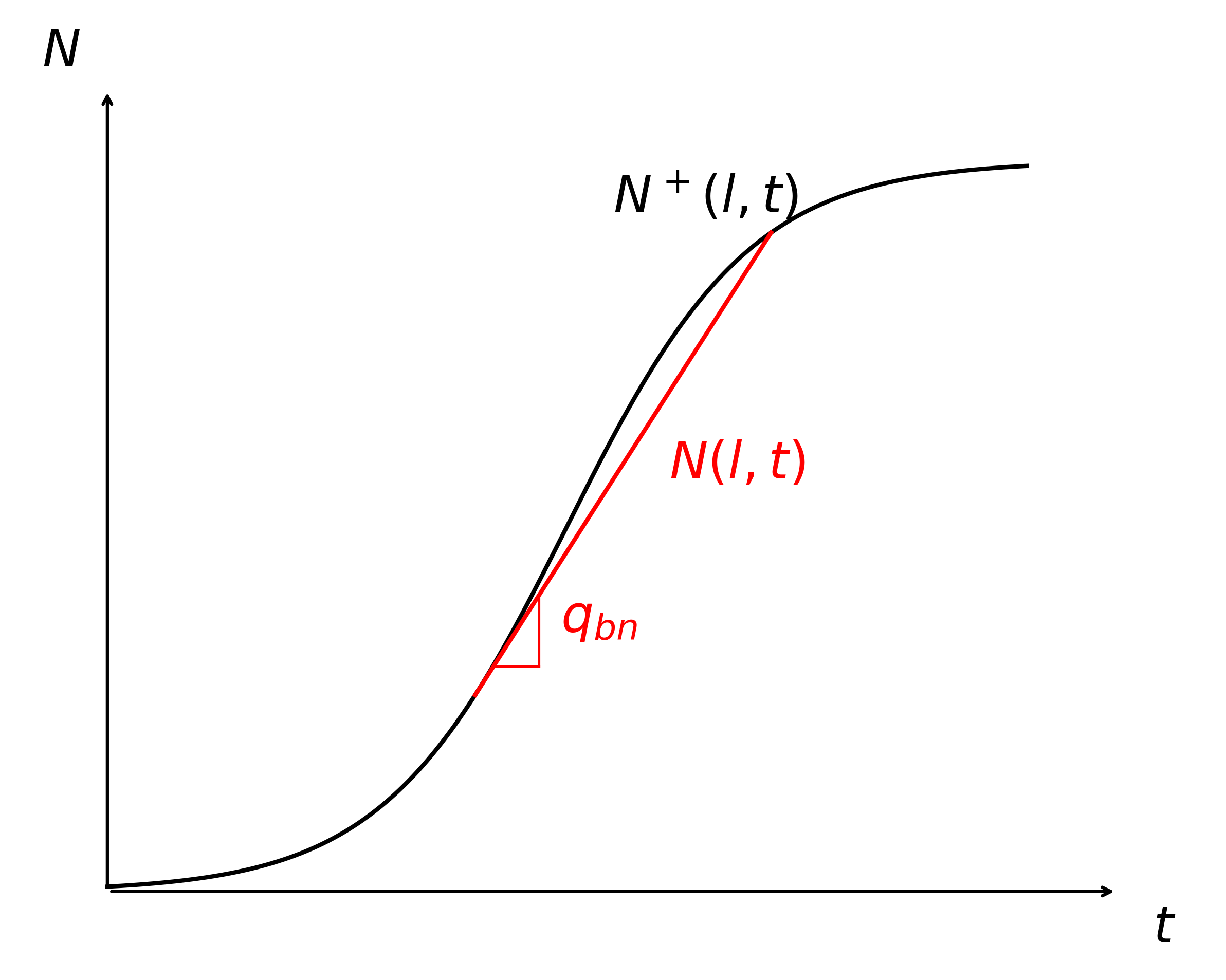}
    \caption{Graphical determination of $N(l,t)$ for given $N^+(l,t)$.}
    \label{fig:queuing}
\end{figure}

\subsection{Solving the Variational Minimization Problem}

In the second stage, we solve the stochastic variational optimization problem \eqref{eq:optimization_functional} subject to the additional constraint:
\[
\mathcal{T}_{\text{avg}}(q_p) \geq \mathcal{T}_{\text{min},\text{avg}}(q_p) \quad \forall q_p \in \text{supp}(\varphi).
\]

To solve \eqref{eq:optimization_functional}, we first discretize the support of the probability distribution, which is then extended to a solution for continuous distributions by considering the corresponding limit for $n \to \infty$:
\[
\text{supp}(\varphi) = \{q_p^0, q_p^1, \dots, q_p^n\}.
\]

For this discrete case, the problem can be formulated as a quadratic optimization problem with constraints using Lagrange multipliers:
\begin{equation}
\label{equation:opt_discrete}
\begin{split}
\mathcal{L}(\mathcal{T}_{\text{avg}}(q_p^1), \dots, \mathcal{T}_{\text{avg}}(q_p^n), & \lambda_1, \dots, \lambda_n) 
= \; \alpha \sum_{i=1}^n p(q_p^i)\, \mathcal{T}_{\text{avg}}(q_p^i) \\
& + (1 - \alpha) \sqrt{ \sum_{i=1}^n p(q_p^i) \left( \mathbb{E}[\mathcal{T}_{\text{avg}}] - \mathcal{T}_{\text{avg}}(q_p^i) \right)^2 } \\
& + \sum_{i=1}^n \lambda_i \left( \mathcal{T}_{\text{min}, \text{avg}}(q_p^i) - \mathcal{T}_{\text{avg}}(q_p^i) \right)
\end{split}
\end{equation}

Without loss of generality, we choose the indices so that the lower bounds are monotonically increasing, i.e., $\mathcal{T}_{\text{avg}}({\text{min}}(q_p^i)) \leq \mathcal{T}_{\text{avg}}({\text{min}}(q_p^{i+1}))$ for all $i$. Since both the mean and standard deviation are convex in their components, the unique critical point of the function is a minimum. The lower bounds of the individual decision variables are always constants; therefore, the optimization program \eqref{equation:opt_discrete} satisfies the linearity constraint qualification. Consequently, the unique minimum of the program satisfies the Karush-Kuhn-Tucker (KKT) conditions.

The KKT conditions of the Lagrangian are as follows:

\begin{enumerate}
\item \textbf{Stationarity:} The partial derivatives of the Lagrangian with respect to $\mathcal{T}_{\text{avg}}(q_p^i)$ must vanish:
\begin{equation}
\label{eq:lag_stat}
\frac{\partial \mathcal{L}}{\partial \mathcal{T}_{\text{avg}}(q_p^i)} = 
\alpha \cdot \frac{\partial \mathbb{E}[\mathcal{T}_{\text{avg}}]}{\partial \mathcal{T}_{\text{avg}}(q_p^i)} 
+ (1-\alpha) \cdot \frac{\partial \text{Std}[\mathcal{T}_{\text{avg}}]}{\partial \mathcal{T}_{\text{avg}}(q_p^i)} 
- \lambda_i = 0, \quad \forall i.
\end{equation}

\item \textbf{Complementary slackness:} The Lagrange multipliers must satisfy:
\begin{equation}
\label{eq:lag_slack}
\lambda_i \cdot \big(\mathcal{T}_{\text{min}}(q_p^i) - \mathcal{T}_{\text{avg}}(q_p^i)\big) = 0, \quad \forall i.
\end{equation}

\item \textbf{Primal Feasibility:} The decision variables must satisfy:
\[
\mathcal{T}_{\text{avg}}(q_i^p) \geq \mathcal{T}_{\min}(q_i^p), \quad \forall i.
\]

\item \textbf{Dual Feasibility:} The Lagrange multipliers must be non-negative: 
\[
\lambda_i \geq 0, \quad \forall i.
\]
\end{enumerate}

Computing the partial derivatives explicitly, we obtain:
\[
\frac{\partial \mathbb{E}[\mathcal{T}_{\text{avg}}]}{\partial \mathcal{T}_{\text{avg}}(q_p^i)} = p(q_p^i),
\]
and using the chain rule:
\[
\frac{\partial \text{Std}[\mathcal{T}_{\text{avg}}]}{\partial \mathcal{T}_{\text{avg}}(q_p^i)} = \frac{p(q_p^i) \big(\mathcal{T}_{\text{avg}}(q_p^i) - \mathbb{E}[\mathcal{T}_{\text{avg}}]\big)}{\text{Std}[\mathcal{T}_{\text{avg}}]}.
\]

Substituting into \eqref{eq:lag_stat} yields:
\[
\alpha \cdot p(q_p^i) 
+ (1-\alpha) \cdot p(q_p^i) \cdot 
\frac{\mathcal{T}_{\text{avg}}(q_p^i) - \mathbb{E}[\mathcal{T}_{\text{avg}}]}{\text{Std}[\mathcal{T}_{\text{avg}}]} 
- \lambda_i = 0, \quad \forall i.
\]

From the dual feasibility conditions, we obtain:
\begin{align}
\label{eq:dual_feasibility}
\lambda_{i} = p_i \cdot \left( \alpha + (1-\alpha)\frac{ \mathcal{T}_{\text{avg}}(q_p^{i}) - \mathbb{E}[\mathcal{T}_{\text{avg}}]}{\text{Std}[\mathcal{T}_{\text{avg}}]}\right) \geq 0.
\end{align}

We now establish that the solution to the discretized optimization problem exhibits an elegant threshold structure, which significantly reduces the computational complexity by transforming the high-dimensional optimization into a simple search for a single threshold value.

\begin{proposition}[Structure of Optimal Solution]
\label{prop:structure}
The optimal solution to the discrete optimization problem \eqref{equation:opt_discrete} exhibits a threshold structure. When ordering the indices by ascending values of $\mathcal{T}_{\min}(q_p^i)$, there exists a critical index $j^*$ such that:
\begin{enumerate}
\item For all $i \leq j^*$: $\mathcal{T}_{\text{avg}}(q_p^i) = \mathcal{T}^*$ for some constant $\mathcal{T}^* > \mathcal{T}_{\min}(q_p^i)$,
\item For all $i > j^*$: $\mathcal{T}_{\text{avg}}(q_p^i) = \mathcal{T}_{\min}(q_p^i)$ (constraint is binding).
\end{enumerate}
Moreover, the threshold value satisfies:
\begin{equation}
\label{eq:Ti_star}
\mathcal{T}^* = \mathbb{E}[\mathcal{T}^*] - \frac{\alpha}{1 - \alpha} \cdot \mathrm{Std}[\mathcal{T}^*].
\end{equation}
\end{proposition}

\begin{proof}
First, we establish that if constraint $i^*$ is binding, then all subsequent constraints must also be binding. Let $\lambda_{i^*} > 0$. From \eqref{eq:dual_feasibility}:
\[
\lambda_{i^*} = p_{i^*} \cdot \left( \alpha + (1 - \alpha) \cdot \frac{\mathcal{T}_{\min}(q_p^{i^*}) - \mathbb{E}[\mathcal{T}_{\text{avg}}]}{\mathrm{Std}[\mathcal{T}_{\text{avg}}]} \right) > 0.
\]

This implies:
\[
\alpha + (1 - \alpha) \cdot \frac{\mathcal{T}_{\min}(q_p^{i^*}) - \mathbb{E}[\mathcal{T}_{\text{avg}}]}{\mathrm{Std}[\mathcal{T}_{\text{avg}}]} > 0.
\]

Since $\mathcal{T}_{\min}(q_p^{i^*+1}) \geq \mathcal{T}_{\min}(q_p^{i^*})$, we have:
\[
\lambda_{i^*+1} \geq p_{i^*+1} \cdot \left(\alpha + (1-\alpha)\frac{\mathcal{T}_{\min}(q_p^{i^*}) - \mathbb{E}[\mathcal{T}_{\text{avg}}]}{\text{Std}[\mathcal{T}_{\text{avg}}]}\right) > 0,
\]
hence $\mathcal{T}_{\text{avg}}(q_p^{i^*+1}) = \mathcal{T}_{\min}(q_p^{i^*+1})$.

Next, we prove that all non-binding entries are equal. Suppose for contradiction that there exist $i, j \leq j^*$ with $\mathcal{T}_{\text{avg}}(q_p^i) > \mathcal{T}_{\text{avg}}(q_p^j)$. Consider the perturbation:
\[
\tilde{\mathcal{T}}(q_p^i) = \mathcal{T}_{\text{avg}}(q_p^i) - \varepsilon \cdot \frac{p_j}{p_i + p_j}, \quad
\tilde{\mathcal{T}}(q_p^j) = \mathcal{T}_{\text{avg}}(q_p^j) + \varepsilon \cdot \frac{p_i}{p_i + p_j}.
\]

This preserves the mean but reduces the standard deviation, yielding $J(\tilde{\mathcal{T}}) < J(\mathcal{T})$, contradicting optimality.

Finally, for non-binding indices where $\lambda_i = 0$, the stationarity condition \eqref{eq:lag_stat} yields:
\[
\alpha + (1 - \alpha) \cdot \frac{\mathcal{T}^* - \mathbb{E}[\mathcal{T}^*]}{\mathrm{Std}[\mathcal{T}^*]} = 0,
\]
which gives the expression for $\mathcal{T}^*$.
\end{proof}

Then, the structure of the optimal solution derived for the discrete case can be extrapolated to the infinite dimensional optimization problem, leading to an elegant reduction to a scalar optimization problem:

\begin{proposition}[Reduction to Scalar Optimization]
\label{prop:scalar}
The variational minimization problem \eqref{eq:optimization_functional} can be reduced to a one-dimensional optimization problem. For a continuous distribution $\varphi$ supported on $[a, b]$, define the modified random variable $\tilde{\mathcal{T}}^{(r)} = \max\{\mathcal{T}, r\}$ for $\mathcal{T} \sim \varphi$. The objective function becomes:
\begin{equation}
J(r) = \alpha \cdot \mathbb{E}[\tilde{T}^{(r)}] + (1 - \alpha) \cdot \mathrm{Std}[\tilde{T}^{(r)}],
\end{equation}
with derivative:
\begin{equation}
\label{eq:scalar_form}
J'(r) = \alpha \cdot P(r) + (1 - \alpha) \cdot \frac{r P(r) + \frac{1}{2} r^2 \varphi(r) - \mu(r) P(r)}{\mathrm{Std}[\tilde{T}^{(r)}]},
\end{equation}
where $P(r) = \mathbb{P}(T \leq r)$ and $\mu(r) = \mathbb{E}[\tilde{T}^{(r)}]$.
\end{proposition}

\begin{proof}
The distribution of $\tilde{\mathcal{T}}^{(r)}$ has density:
\[
f_{\tilde{\varphi}(r)}(x) = \varphi^{\leq}(r) \cdot \delta_r(x) + \varphi(x) \cdot \mathbf{1}_{(r, \infty)}(x),
\]
where $\varphi^{\leq}(r) = \int_{-\infty}^r \varphi(x) \, dx$.

Define:
\begin{align*}
P(r) &= \int_a^r \varphi(x) \, dx, \\
M_1(r) &= \int_r^b x \cdot \varphi(x) \, dx, \\
M_2(r) &= \int_r^b x^2 \cdot \varphi(x) \, dx.
\end{align*}

Then:
\begin{align*}
\mu(r) &= r \cdot P(r) + M_1(r), \\
\mathrm{Var}[\tilde{T}^{(r)}] &= r^2 P(r) + M_2(r) - \mu(r)^2.
\end{align*}

Differentiating:
\begin{align*}
\mu'(r) &= P(r), \\
\mathrm{Var}'(r) &= 2r P(r) + r^2 \varphi(r) - 2 \mu(r) P(r).
\end{align*}

Therefore:
\[
\frac{d}{dr} \mathrm{Std}[\tilde{T}^{(r)}] = \frac{\mathrm{Var}'(r)}{2 \cdot \mathrm{Std}[\tilde{T}^{(r)}]},
\]
yielding the expression for $J'(r)$.
\end{proof}

\begin{corollary}[Optimal Control Function]
The optimal control function for the variational optimization problem takes the form:
\begin{equation}
\label{eq:control_function}
T(x) = \max\{r^*, T_{\min}(x)\},
\end{equation}
where $r^*$ is the unique solution to $J'(r) = 0$. This reduction transforms an infinite-dimensional optimization problem into a smooth one-dimensional root-finding problem, solvable using elementary numerical methods.
\end{corollary}

The same approach extends naturally to the discrete case. When the support of the distribution is finite, the objective function remains piecewise smooth and continuous, with non-differentiable points only at the breakpoints of $T_{\min}(x_i)$. The optimal solution retains the threshold structure:
\[
T(x_i) = \max\{r^*, T_{\min}(x_i)\}.
\]
The scalar function $J(r)$ can be efficiently minimized by evaluating it over the sorted list of breakpoints, enabling an efficient solution of the discrete problem.

\subsection{Computation of Optimal Speed Limits}
To operationalize the target total travel time derived in the previous subsection, we now introduce a Model Predictive Control (MPC) framework to compute time-varying speed limits that steer the actual system performance toward this reference. The MPC paradigm is especially suitable for the real-time control of dynamic systems, as it continuously updates control actions based on current system measurements while optimizing over a receding horizon. At every time step, the MPC solves a constrained optimization problem using the latest traffic state, implements the control action for the current time interval, and then shifts the prediction horizon forward. This architecture inherently forms a closed feedback loop, which enables more precise traffic flow control compared to an open-loop system. 

Our objective is to minimize the deviation between the actual average travel time and the previously computed goal value, denoted \( \mathcal{T}^*_{\text{avg}} \). The scalar control variable is denoted by \( \mathbf{v}_c  \). The Model Predictive Control optimization problem at time \( t_0 \) is then formulated as:
\[
\min_{\mathbf{v}_c \in [v_{\min}, v_{\max}]} \left\| \mathcal{T}_\text{avg}(t_0; \mathbf{v}_c) - \mathcal{T}^*_\text{avg} \right\|_1 
\quad \text{subject to system constraints}.
\]
While control-theoretic approaches typically use the $L^2$ norm of the deviation for analytical convenience, we adopt the physically more meaningful $L^1$ norm. 

For the proposed control architecture to be feasible, we assume that the VSL location is positioned at the upstream end of the corridor, i.e., at \( x = 0 \). The inflow into the corridor at location \( x = 0 \) follows a point queue discipline, where vehicles that cannot immediately enter due to limited supply are stored in an abstract upstream queue. Let \( \ell(t) \) denote the length of the point queue at time \( t \). The evolution of the queue is governed by the conservation law
\begin{equation}
\label{eq:pq_conservation}
\frac{d\ell(t)}{dt} = D(t) - q(0,t),
\end{equation}
where the actual inflow \( q(0,t) \) is given by the state-dependent expression
\begin{equation}
\label{eq:pq_demsup}
q(0,t) =
\begin{cases}
s(k(0,t)), & \text{if } \ell(t) > 0, \\
\min\left\{ s(k(0,t)), D(t) \right\}, & \text{if } \ell(t) = 0.
\end{cases}
\end{equation}

\paragraph{State variables and dynamics.}
Let the horizon be discretized into $N$ equal time–steps of length $\Delta t = 1\,\text{min}$. Denote by
\begin{align*}
L_k &\;[\text{veh}] &\text{queue length at the link entrance,}\\
D_k &\;[\text{veh·h}] &\text{cumulative total delay up to step }k,\\
V_k &\;[\text{veh}] &\text{cumulative vehicles admitted up to step }k,
\end{align*}
where $k = 0,\dots,N$ and $L_0 = D_0 = V_0 = 0$.  
The queue evolves according to
\[
L_{k+1} = \max\!\bigl\{L_k + u_k - q_{bn}\Delta t,\; 0\bigr\},
\]
with $u_k$ the vehicles admitted during step~$k$. The total delay now consists of two components:
\[
D_{k+1} = D_k
        + L_k\,\Delta t
        + \frac{u_k}{3600} \left( \tau_k(u_k, d_k) - T_0 \right),
\qquad
V_{k+1} = V_k + u_k,
\]
where $\tau_k(u_k, d_k)$ is the expected travel time on the link in minutes for vehicles admitted at step $k$, and $T_0$ is the free-flow travel time. The term $\frac{u_k}{3600} (\tau_k - T_0)$ accounts for the additional link traversal delay due to reduced speed, which is calculated as the inverse of the control speed $v_k$.

\paragraph{Delay‑budget trajectory}
The total excess delay that still satisfies the target average travel time~$T_{\text{tar}}$ is
\[
D_{\text{tot}} = \frac{T_{\text{tar}} - T_0}{60}\,V_N^{\max},
\qquad
V_N^{\max} = \sum_{k=0}^{N-1} d_k,
\]
where $d_k$ is the exogenous demand in veh per step.  
We spread this “delay budget" over the horizon via
\begin{equation}
D_k^{\max} = D_{\text{tot}} \left(\tfrac{k}{N}\right)^{\gamma},
\qquad
\gamma > 0,
\end{equation}
where the exponent $\gamma$ controls the temporal distribution of delay tolerance:  
larger values of $\gamma$ concentrate more of the delay budget to vehicles arriving near the peak of demand, smaller values give relatively more slack to off-peak arrivals near the beginning or end of the horizon.

\paragraph{Per‑step optimisation problem (MPC view)}
Given the current state $(L_k, D_k)$ and demand $d_k$, choose the largest admissible inflow
\begin{equation}
u_k^\star = \underset{u \in \mathcal U_k}{\arg\max}\; u
\quad\text{s.t.}\quad
\tilde D_{k+1}(u) \le D_{k+1}^{\max},
\end{equation}
where
\begin{equation}
\mathcal U_k = [u^{\min},\; u^{\max}],\qquad
u^{\min} = 5800\,\Delta t,\quad
u^{\max} = \min\{d_k,\; C\Delta t\},
\end{equation}
and
\[
\tilde D_{k+1}(u) = D_k + L_k\,\Delta t
+ \frac{u}{3600} \left( \tau_k(u, d_k) - T_0 \right)
+ \max\{L_k + u - C\Delta t,\, 0\} \Delta t
\]
is the predicted delay one step ahead if metering rate~$u$ is applied.  
Note that $\tilde D_{k+1}(u)$ remains monotonic in~$u$, so problem~(2) reduces to a binary feasibility test:
\begin{equation}
u_k^\star =
\begin{cases}
u^{\max}, & \text{if } \tilde D_{k+1}(u^{\max}) \le D_{k+1}^{\max},\\[6pt]
u^{\min}, & \text{otherwise}.
\end{cases}
\end{equation}

\paragraph{Speed limit derivation}
The MPC algorithm determines the metering rate~$q_m = u_k^\star/\Delta t$, not the speed limit directly. To translate the metering rate into a speed limit at the upstream end of the link, we determine the speed as
\[
\mathbf{v}_{c,k} = \frac{Q(k^+)}{k^+},
\qquad
\text{where } k^+ \text{ is the largest value satisfying } Q(k^+) = q_m.
\]

\paragraph{Control implementation}
Algorithm (3) is executed every minute. If the instantaneous demand~$d_k$ is below $u^{\min}$, then $u_k^\star = d_k$ (no throttling is needed).  
The resulting metering trajectory ensures the travel-time target is met, while guaranteeing that the metering rate does not fall below the soft floor $u^{\min}$, except when demand is insufficient.

\subsection{Calibration of Model Parameters}

In this section, we conduct numerical calculations to evaluate the proposed control algorithm. Parameter selection is based on empirical data from a 9.0~km section of Interstate 880 northbound in the San Francisco metropolitan area, just upstream of the Washington Avenue off-ramp bottleneck. Data were collected from 18 detectors along the corridor. The road geometry and congestion formation mechanisms are detailed in~\cite{Munoz2002}, and the calibration method follows ~\cite{Hammerl2024, Hammerl2025}.

For the estimation of upstream demand, we focus our parameter estimation on the evening rush hour (6:30-9:30 PM). To estimate the upstream boundary condition, we first approximate time-flow measurements from the upmost detector using the following deterministic piecewise linear function:
\[
f(t) = 
\begin{cases}
at + b & \text{for } t \leq t_p, \\
ct + d & \text{for } t > t_p,
\end{cases}
\]
where \( t \) represents hours after 6:30~PM (e.g., \( t = 2.5 \) for 9~PM), subject to the following constraints:
\begin{align*}
a &\geq 0 \quad \text{(non-negative slope in first segment)}, \\
b &\geq 0 \quad \text{(non-negative intercept)}, \\
c &\leq 0 \quad \text{(negative slope in second segment)}, \\
at_p + b &= ct_p + d \quad \text{(continuity at breakpoint \( t_p \))}.
\end{align*}

The parameters \( (a, b, c, d, t_p) \) were simultaneously estimated using Sequential Least Squares Programming (SLSQP):
\begin{equation}
\label{eq:det_up}
f(t) = 
\begin{cases}
447.23t + 5795.46 & \text{for } t \leq 1.5, \\
-620.37t + 7708.81 & \text{for } t > 1.5.
\end{cases}
\end{equation}

To model the stochasticity of the boundary flow, we consider the daily peak value of the upstream boundary flow $q(0,t)$ as a normally distributed random variable $\mathcal{N}(\mu, \sigma^2)$. From the standard deviation of the daily maximum flows of the most upstream detector, a value of 191.17 was estimated for this parameter of the upstream boundary flow. For the model parameters, we use the values $\mu = 6620$ and $\sigma = 191$. We estimate the resulting stochastic function for the upstream flow such that the flow \( q(0,t) \) matches the deterministic average values from Equation~\ref{eq:det_up} at times \( t = 0 \) and \( t = 5 \), that is, we impose the condition that \( q(0,0) = 5571.84 \) and \( q(0,4) = 5227.33 \) hold with probability one. This yields the following formula:
\begin{equation}
\label{eq:stoch_up}
f(x) =
\begin{cases}
\displaystyle \left( \frac{q_p - 5795.46}{1.5} \right) x + 5571.84 & \text{if } 0 \leq x \leq 1.5, \\[10pt]
\displaystyle \left( \frac{5227.33 - q_p}{1.5} \right) x + \left( q_p - 2 \cdot \frac{5847.7 - q_p}{1.5} \right) & \text{if } 1.5 < x \leq 3.
\end{cases}
\end{equation}
where $q_p \sim \mathcal{N}(6620, 191^2)$.

The length of the simulated corridor is selected such that vehicles traverse it at free-flow speed in exactly 5 minutes, corresponding to a length of 9.3 kilometers.
The bottleneck capacity is determined using measurements from two representative detectors centrally positioned within the corridor. Average flow and occupancy were calculated for each 5-minute interval during morning peak periods, and the resulting curve was smoothed using a Savitzky-Golay filter. Figure \ref{fig:estimate_bn_flow} presents the processed flow-occupancy relationship. The characteristics clearly show that both detectors are significantly affected by the queue. We selected the simulated bottleneck capacity as the flow at the rightmost point of the smoothed curve, corresponding to a rounded value of \( q_{\text{bn}} = 6240 \). For the intensity of capacity drop, we chose a value of $\Delta=10 \%$.
\begin{figure}[H]
   \centering
   \includegraphics[width=0.8\textwidth]{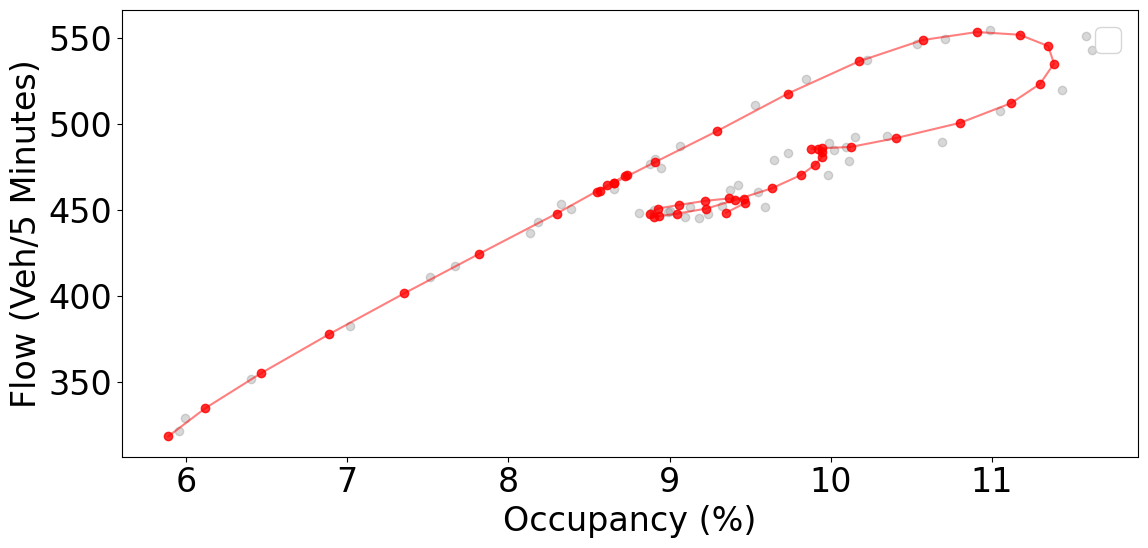} 
   \caption{Flow-density relationship averaged from two central detectors. Raw data points in grey, smoothed curve in red.}
   \label{fig:estimate_bn_flow}
\end{figure}
A triangular shape is assumed for the fundamental diagram. We determine the free-flow speed by converting the speed limit of \SI{70}{\mile\per\hour} to a rounded value value of \SI{112}{\kilo\meter\per\hour}. The jam density is calculated from the minimum number of lanes, \(3\), multiplied by the inverse of the average vehicle length (we assume a typical value of \SI{7.1}{\meter}), resulting in a rounded value of \SI{420}{\per\kilo\meter}. The critical density of the road section cannot be derived directly from the collected data. Instead, we calculate it using the ratio between the two slopes of the fundamental diagram branches. We assume that this ratio is (-5), which is a common choice in traffic flow literature. This yields a critical density of \SI{70}{\per\kilo\meter} and a maximum flow of \(112 \cdot 70 = 7840\) vehicles per hour. The initial vehicle distribution is chosen such that the number of vehicles in each cell corresponds to the boundary flow of the first simulated time step. The travel time function is derived from the speed–flow relation by the formula
\begin{equation}
\label{eq:meter_to_vsl}
v_k(u_k, d_k) =
\begin{cases}
112, & \text{if } d_k \le u_k,\\[6pt]
\displaystyle \frac{u_k}{\left( \frac{350}{7640}(7640 - d_k) + 70 \right)},
& \text{otherwise},
\end{cases}
\end{equation}
with $\tau_k(u_k, d_k) = \ell / v_k(u_k, d_k) \cdot 60$, where $\ell$ is the link length in km.

\subsection{Simulation Results and Analysis}

Based on the empirically calibrated parameters, we conduct two simulation suites: The first suite is based on the simplified assumption that the time-dependent upstream demand follows exactly the distribution resulting from Equation \ref{eq:stoch_up} and that only a single variable speed limit is positioned at point x=0 of the corridor. The algorithm proceeds as follows: First, we calculate the minimum achievable average travel time through calibration of the VSL. For this purpose, we make the simplifying assumption that the queue completely dissipates by the end of the simulation period (t=5), which corresponds to the condition $q_p \leq 10152$. For compactness, we introduce $a= \SI{5571.84}{vehicles\per\hour}$, $b= \SI{4606.96}{vehicles\per\hour}$, $\alpha(q_p) = \frac{q_p - a}{2}$ $, \beta(q_p) = \frac{q_p - b}{3}$ and $H(q_p) = q_p - C$. The cumulative arrival flow at time t, $A(t)$, is defined as $A(t)=\int_0^t q(0,t)dt$. We first calculate $T_{avg}$ as a function of $q_p$. For $q_p \leq 6240$, obviously $T_{avg}(q_p) = \SI{5}{min}$. For \( q_p \geq 6240 \), let \( t_0 = (q_{\text{bn}} - a)/\alpha \) denote the time at which the upstream inflow first exceeds the bottleneck capacity. The time \( t_1 \) at which the last vehicle affected by the queue exits the upstream boundary satisfies \( A(t_1) - A(t_0) = q_{\text{bn}} (t_1 - t_0) \), yielding \( t_1 = 2 + 2\sqrt{L_2 / (q_p - b)} \), where \( L_2 = H^2/2 \cdot (1/\alpha + 1/\beta) \) is the maximum queue length. The total delay due to queuing (in hours) is

\begin{align*}
\mathcal{T}_{\text{queue}} 
&= \int_{t_0}^{t_1} \left[ A(t) - \left( q_{\text{bn}} (t - t_0) + A(t_0) \right) \right] \, dt \\
&= 
\frac{\alpha}{6} (2 - t_0)^3 + L_2 s_1 + \frac{H}{2} s_1^2 - \frac{\beta}{6} s_1^3,
\end{align*}

where $ s_1=t_1-2$. The resulting average travel time per vehicle (in minutes) is
\begin{equation}
\label{eq:avg_tt_s1}
\mathcal{T}_{avg} = 5 + \frac{60 \, \mathcal{T}_{\text{queue}}}{A(0,4)} = 5 + \frac{60 \, \mathcal{T}_{\text{queue}}}{2q_p + a + b}.
\end{equation}
The probability distribution of $\mathcal{T}_{avg}$ is obtained from equation \ref{eq:avg_tt_s1} through application of the change-of-variable rule:
\[
p_{\mathcal{T}_{\text{avg}}}(t) = p_{q_p}\left( \mathcal{T}_{\text{avg}}^{-1}(t) \right) \left| \frac{d}{dt} \mathcal{T}_{\text{avg}}^{-1}(t) \right|.
\]
The critical value of $\alpha$, defined as the minimum value above which the optimal distribution exactly corresponds to the lower bound, is determined by differentiating the objective function~\ref{eq:optimization_functional} with respect to $\mathcal{T}_{\text{avg}}$ and substituting $\mathcal{T}_{\text{avg}} = 5$ minutes. Using the abbreviations $M_1 := \mathbb{E}[\mathcal{T}_{\text{min,avg}}]$ and $\sigma_0 := \text{Std}[\mathcal{T}_{\text{min,avg}}]$, we obtain 
\begin{equation}
\label{eq:crit_alpha}
\alpha_{\text{crit}} = \frac{M_1 - 5}{M_1 - 5 + \sigma_0} \approx 0.57.
\end{equation}
For $\alpha < \alpha_{\text{crit}}$, we determine the optimal lower bound $\tau^*$ by solving equation~\ref{eq:scalar_form}. Table~\ref{table:joint_metrics} shows $\tau^*$ for representative $\alpha$-values. The control policy achieves substantial performance improvements of $44\%$ for high $\alpha$ values and ca. $30\%$ for low $\alpha$ values, demonstrating its effectiveness across different operating regimes.
Figure 4 illustrates the VSL (Variable Speed Limit) profile at $\alpha = 0.5$ with peak flow values of 6,500 and 6,800 respectively.

\begin{figure}[h!]
   \centering
   \includegraphics[width=0.8\textwidth]{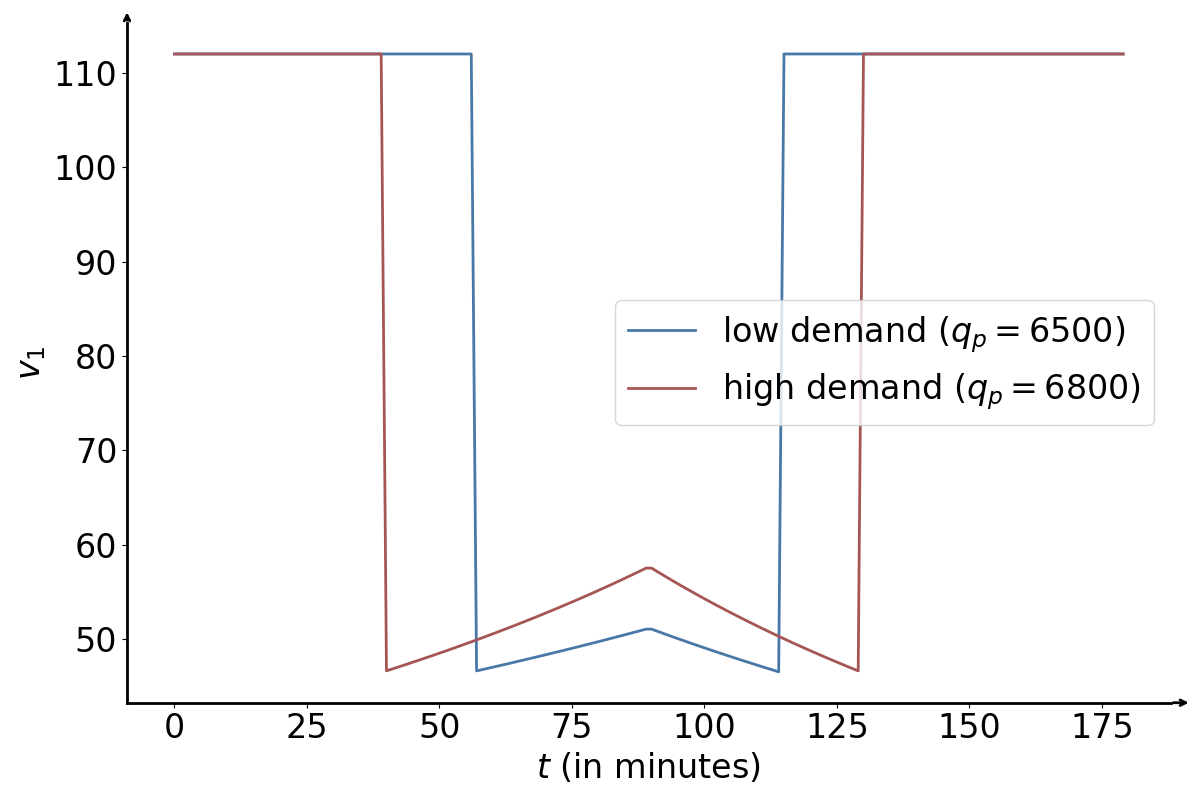}
   \caption{Comparison of VSL profiles.}
   \label{fig:compare_vsl_peak_flow}
\end{figure}

\begin{table}[h!]
\centering
\begin{tabular}{cccccccc}
\hline
$\alpha$ & $\tau^*$ & $q_p^*$ & $\mathbb{E}[\tau]$ & $\operatorname{Std}[\tau]$ & $J_{\min}$ & $\Delta J $ & Rel.\ improvement \\
\hline
$\geq 0.57$ & 5.00 & N/A   & 5.69 & 0.637 & 3.67 & -2.91 & -44.2\% \\
0.5         & 5.44 & 6525  & 5.70 & 0.632 & 3.16 & -2.62 & -39.8\% \\
0.4         & 6.08 & 6667  & 5.81 & 0.539 & 2.65 & -2.35 & -35.7\% \\
0.3         & 6.77 & 6788  & 6.11 & 0.380 & 2.10 & -2.10 & -31.9\% \\
0.2         & 7.50 & 6909  & 6.60 & 0.218 & 1.49 & -1.92 & -29.2\% \\
\hline
\end{tabular}
\caption{Joint performance metrics as a function of $\alpha$. the last two columns give absolute and relative improvements over a no-control baseline.}
\label{table:joint_metrics}
\end{table}

\section{Real-Time Capacity Variations}
In the previous sections, we analyzed the optimization of travel time reliability in a day-to-day context under stochastic peak period inflow rates. %the utility of reliability-conscious drivers who chose trips without access to real-time daily traffic information. 
In this section, we instead consider a model in which travel time variations are primarily caused by instantaneous random capacity changes at the downstream bottleneck and minimize the cost functional \begin{equation}
\label{eq:cost_funct_ou}
    J = \int_0^T q(0,t) \cdot \left[\alpha \mathbb{E}[\tau(t)] + (1-\alpha) \mathrm{Std}(\tau(t))\right] dt
\end{equation}. In this section, we assume that upstream boundary flow behaves deterministically according to formula \ref{eq:det_up}. Again, we assume that a point queue \( P(t) \) buffers excess demand following the state equations \ref{eq:pq_conservation} and \ref{eq:pq_demsup}, and that the traffic evolution on the link follows the triangular fundamental diagram \ref{eq:det_up}. A vehicle that enters the spatial queue at \( t \) experiences cruise time \( \tau_c(t) = \ell / v_{\mathrm{vsl}}(t) \) and an additional delay \( \tau_q(t) \) until the cumulative discharge clears its backlog. Total travel time is therefore
\begin{equation}
T(t) = \max\left\{ \tau_c(t), \tau_q(t) \right\}.
\end{equation}

To model the stochastic evolution of bottleneck capacity, we employ an Ornstein-Uhlenbeck (OU) process, which provides several key advantages for modeling real-world capacity fluctuations: its mean-reverting property ensures temporary disruptions (minor incidents, driver variations) naturally dissipate rather than causing unrealistic long-term deviations; the upper bound prevents physically unrealistic capacity spikes that would instantaneously clear peak congestion; and its continuous sample paths avoid large instantaneous jumps that would propagate as overly pessimistic predictions in the control algorithm.

\subsection{MPC Formulation}

\subsubsection{Control Strategy}

The metering inflow at the upstream boundary is chosen as
\begin{equation}
q_m(t)\;=\;
\min\bigl\{d(t),\;C(t)-K\,S(t)\bigr\},           \label{eq:qmlaw}
\end{equation}
where $d(t)$ is the arrival demand, $S(t)$ the queue length at the downstream bottleneck,
and $K>0$ a scalar feedback gain.
The obtained metering rate is converted to a posted speed limit via
the static relation \ref{eq:meter_to_vsl} and is bounded below by $v_{\min}=40\;\mathrm{km/h}$; if
$q_m(t)\ge d(t)$ the sign displays the free‐flow speed
$v_{\mathrm{f}}=112\;\mathrm{km/h}$.
A kinematic FIFO constraint
\(
l/v_c(t+\!\Delta t)\;\ge\;l/v_c(t)-\Delta t
\)
(with link length $l=9.34\;\mathrm{km}$) prevents backward moving shocks at the entrance. For every minute step $k$ let
$T_k$ denote the random total travel time constructed from
$\tau_q$ in~\eqref{eq:tauqMom} and the cruise time
$l/v_c(t_k)$.
The horizon cost is
\begin{equation}
J(K) \;=\;
\sum_{k=0}^{N-1}
d_k\,\Bigl(
      \alpha\,\mathbb{E}[T_k]+
      (1-\alpha)\,\mathrm{Std}[T_k]
     \Bigr)\,\Delta t.
\label{eq:cost}
\end{equation}
A uniform grid of $n=60$ points in $[K_{\min},K_{\max}]$
yields step‑free motion of~$J$ and executes in
$\mathcal{O}(10^{-1}\text{ s})$ on standard hardware. All constraints (queue dynamics, capped capacity, FIFO condition)
remain linear or second‑order cone, so the resulting SMPC is solved by OSQP/SCS in well under $0.2\text{ s}$ for a 3\,h horizon.

\subsubsection{Travel Time Calculation}

The discharge capacity at the downstream bottleneck is described by a \emph{capped Ornstein–Uhlenbeck} (OU) process, which evolves according to the stochastic differential equation
\begin{align}
d\widetilde{C}_t &= \kappa\bigl(c_{\max}-\widetilde{C}_t\bigr)\,dt
                   + \sigma\,dW_t,                               \label{eq:OU_free}\\[2pt]
C_t              &= \min\bigl\{\widetilde{C}_t,c_{\max}\bigr\},  \label{eq:OU_cap}
\end{align}
where $c_{\max}=6\,240\;\mathrm{veh/h}$ serves simultaneously
as the long-run mean and the upper reflecting barrier.
Equation~\eqref{eq:OU_free} models mean-reverting Gaussian fluctuations; the projection in~\eqref{eq:OU_cap} eliminates unrealistically high capacity spikes observed with an unconstrained OU.  For a vehicle entering the link at \(t_0\), denote
\[
M(u)\;=\;\int_{t_0}^{t_0+u} C_s\,ds
\]
the cumulative discharge over the interval \([t_0,t_0+u]\).
Write $a:=c_{\max}-C_0$ and
\(
\delta:=\sigma/\sqrt{4\pi\kappa},
\)
where $\kappa$ is the mean-reversion rate.
To obtain the pointwise mean of the capped OU, let the \emph{uncapped} OU at lag $s\ge0$ be (see e.g. \cite{Oksendal2003})
\[
\widetilde C_{t_0+s}\sim\mathcal N\!\bigl(m_s,v_s\bigr),
\quad
m_s=c_{\max}-a\,e^{-\kappa s},\;
v_s=\frac{\sigma^{2}}{2\kappa}\bigl(1-e^{-2\kappa s}\bigr),
\quad
a:=c_{\max}-C_0.
\]
Put
\(
d_s:=(c_{\max}-m_s)/\sqrt{v_s}
\)
and denote by $\Phi$ and $\varphi$ the standard normal CDF and PDF.

\noindent
By the law of total expectation, we split the event
$\{\widetilde C_{t_0+s}<c_{\max}\}$ from its complement:
\begin{align}
\mathbb E[C_{t_0+s}]
 &= \mathbb E\!\bigl[\widetilde C_{t_0+s}\,\mathbf 1_{\{\widetilde C_{t_0+s}<c_{\max}\}}\bigr]
    \;+\;
    c_{\max}\,\mathbb P\!\bigl(\widetilde C_{t_0+s}\ge c_{\max}\bigr)\\
 &= \underbrace{m_s\,\Phi(d_s)
      -\sqrt{v_s}\,\varphi(d_s)}_{\text{truncated-normal moment}}
    \;+\;
    c_{\max}\bigl[1-\Phi(d_s)\bigr].
\end{align}

\noindent
Rearranging and recalling that
$m_s=c_{\max}-a\,e^{-\kappa s}$ gives
\begin{equation}
\label{eq:C_mean_final}
\mathbb E[C_{t_0+s}]
  = c_{\max}
    - a\,e^{-\kappa s}\,\Phi(d_s)
    - \sqrt{v_s}\,\varphi(d_s).
\end{equation}
Equation~\eqref{eq:C_mean_final} reduces to $c_{\max}$ when
$a\to0$ (i.e.\ $C_0\to c_{\max}$ or $s\to\infty$) and to the uncapped OU mean $m_s$ when $\sigma\to0$, as expected.
Carrying out the integration in $s$ gives the closed-form approximation
\begin{equation}
\label{eq:EMuExact}
\mathbb E[M(u)]
      = c_{\max}u
        -\frac{a}{\kappa}\bigl(1-e^{-\kappa u}\bigr)
        -\delta u
        +\frac{\delta}{4\kappa}\bigl(1-e^{-2\kappa u}\bigr)
\end{equation}

Because $Y_t$ is obtained through the non-linear map $x\mapsto \min(x,\mu)$, its bivariate distribution is a doubly (upper-)truncated Gaussian.  Writing $Z_t := (X_t-\mu)\sqrt{2\kappa}/\sigma\sim\mathcal N(0,1)$
and $U_t := \min(0,Z_t)$, one finds
\begin{equation}
    \operatorname{Cov}[C_{t_0},C_{t_0+\tau}] 
    = \frac{\sigma^{2}}{2\kappa}\,h\!\bigl(e^{-\kappa\tau}\bigr),
    \quad 0\le h(\rho)\le \bigl(1-2/\pi\bigr)\rho.
    \label{eq:cov_clipped_kernel}
\end{equation}
The explicit form of $h$ involves the bivariate normal \emph{cdf} but is
omitted here for brevity.  Plugging~\eqref{eq:cov_clipped_kernel} into the double integral
\begin{equation}
    \operatorname{Var}\bigl[M(u)\bigr]
    = \iint\limits_{[0,u]^2}\!\operatorname{Cov}\bigl[C_{t_0+s},C_{t_0+t}\bigr]
      \,\mathrm ds\,\mathrm dt
\end{equation}
leads, after the change of variables $r=e^{-\kappa|t-s|}$, to the
representation
\begin{equation}
    \operatorname{Var}\bigl[M(u)]
    = \frac{\sigma^{2}}{\kappa^{2}}
      \int_{0}^{1}\frac{h(r)}{r}\Bigl[1-e^{-\kappa u}(1+r) + e^{-2\kappa u}r\Bigr]\,\mathrm dr.
    \label{eq:var_SY_exact}
\end{equation}
Since the average travel time is usually much higher than the inverse mean-reversion rate of the bottleneck capacity, \(u\gg1/\kappa\), the exponentials in~\eqref{eq:var_SY_exact} become negligible and the integral simplifies to the linear growth law
\begin{equation}
\label{eq:VarMu}
    \operatorname{Var}\bigl[M(u)\bigr]
           \;\approx\;K\,\frac{\sigma^{2}}{\kappa^{2}}\,u,
           \qquad K\;\approx\;0.292.
\end{equation}
(The numerical constant $K$ follows from $K=\int_{0}^{1}h(r)/r\,dr$.)

Since \(u\gg1/\kappa\), the horizon \(u\) covers many \emph{reversion times} of the OU, so the covariance kernel in~\eqref{eq:cov_clipped_kernel} has essentially
decayed to zero long before the integral terminates.  A functional central-limit theorem (Kipnis–Varadhan, 1986; {Oksendal2003}) therefore converts the centered process
\(M(u)-\mathbb E[M(u)]\) into a Brownian motion with variance rate $\sigma^{2}/\kappa^{2}$, justifying the drift–diffusion approximation used below.

Let $B=N(0,t_0)$ be cumulative flow at the entry of the link at time $t_0$. The queue‐delay component of the travel time is the first-passage time
\begin{equation}
\tau_q \;=\;
\inf\{u\ge0 : M(u)\ge B\}.
\end{equation}

The relation ($u > 2 \cdot \theta$) typically applies in horizons of practical interest, hence the waiting-time \(u\) covers several reversion times of the OU the integral $M(u)$. Consequently, $M(u)$ behaves approximately as a Brownian motion with constant drift $c_{\max}-\delta$ and linear variance growth $\operatorname{Var}[M(u)] \approx \sigma^{2}u$
(cf.~\eqref{eq:VarMu}).  \citet{Wald1947} showed that the hitting time of such a drift–diffusion process
to a fixed threshold~${B}$ follows an inverse-Gaussian distribution
\[\mathrm{IG}\!\left(\smash{\frac{B}{c_{\max}-\delta}},\smash{\frac{B^{2}}{\sigma^{2}}}\right).\]
Hence
\begin{equation}
   \mathbb E[\tau_q] = \frac{B}{c_{\max}-\delta}, 
   \qquad
   \operatorname{Var}[\tau_q] \approx \frac{B\,\sigma^{2}}{(c_{\max}-\delta)^{3}}.
   \label{eq:tauqMom}
\end{equation}
Equations~\eqref{eq:EMuExact}-\eqref{eq:tauqMom} enter directly into the
stochastic MPC cost without further approximation.
\citet{Cassidy1999} report empirical observations of oscillatory capacity fluctuations at a freeway bottleneck under congested conditions, with a typical period of approximately 5 minutes and an amplitude of about 5\% of the uncongested bottleneck capacity. To reflect these characteristics, we calibrated the parameters of the stochastic process to $\kappa = 0.2\;\mathrm{min}^{-1}$, $\mu = 6240\;\mathrm{veh/h}$, and $\sigma = 139.5$.

\subsection{Numerical Experiments}

Let \( K^\star \) denote the gain that minimizes the chosen objective function \( J(K) \) over the tested grid. Typical outcomes for the optimal gain are:

\begin{center}
\begin{tabular}{c|c|c|c}
$\alpha$ & $r$ & $K^\star$ [h$^{-1}$] & VSL window (min) \\\hline
$0.75$ & 0.25 & $\approx0.022$ & $\sim100$, floor $50$--$55$ km/h\\
$0.50$ & 0.50 & $\approx0.045$ & $\sim125$, floor $40$--$45$ km/h
\end{tabular}
\end{center}

For a trapezoidal upstream demand following formula \ref{eq:det_up}, figures \ref{figure:high_alpha} and \ref{figure:low_alpha} visualize the temporal evolution of the variable speed limit and the expected travel time at departure for a specific realization of the time-dependent bottleneck capacity.
A larger standard‑deviation weight elongates the active
speed‑limit phase and deepens the restriction, while a
mean‑dominated objective relaxes control earlier. 

\begin{figure}[H]
  \centering
  \begin{subfigure}[b]{0.49\textwidth}
    \includegraphics[width=\textwidth]{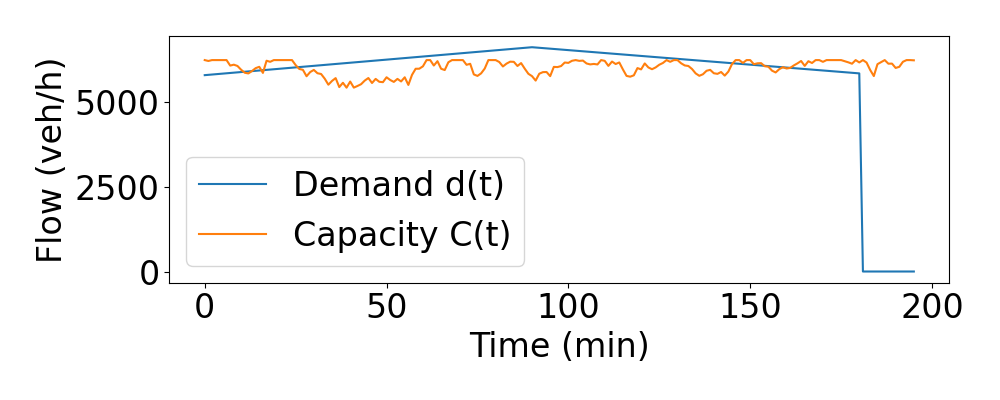}
    \label{fig:sub1}
  \end{subfigure}
  \hfill
  \begin{subfigure}[b]{0.49\textwidth}
    \includegraphics[width=\textwidth]{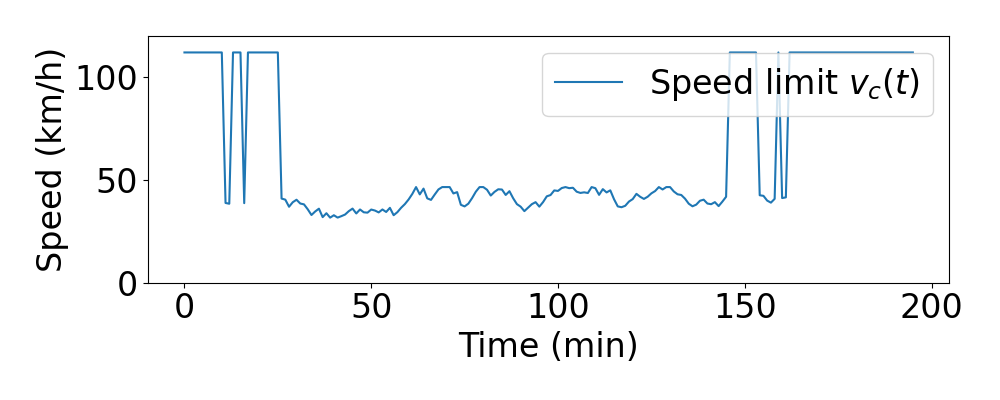}
    \label{fig:sub2}
  \end{subfigure}
  \vskip 1em
  \begin{subfigure}[b]{0.65\textwidth} % wider bottom figure
    \includegraphics[width=\textwidth]{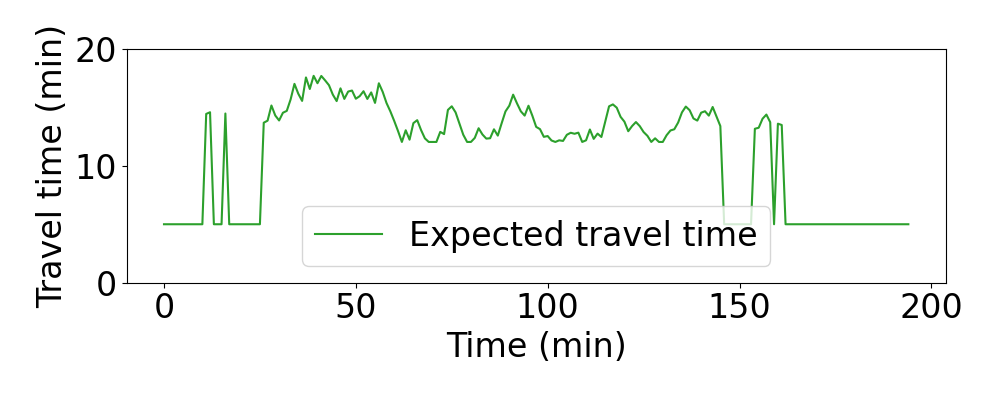}
    \label{fig:sub3}
  \end{subfigure}
  \caption{Stochastic evolution of capacity, VSL, and travel time, $\alpha=0.75$.}
  \label{figure:high_alpha}
\end{figure}

\begin{figure}[H]
  \centering
  \begin{subfigure}[b]{0.49\textwidth}
    \includegraphics[width=\textwidth]{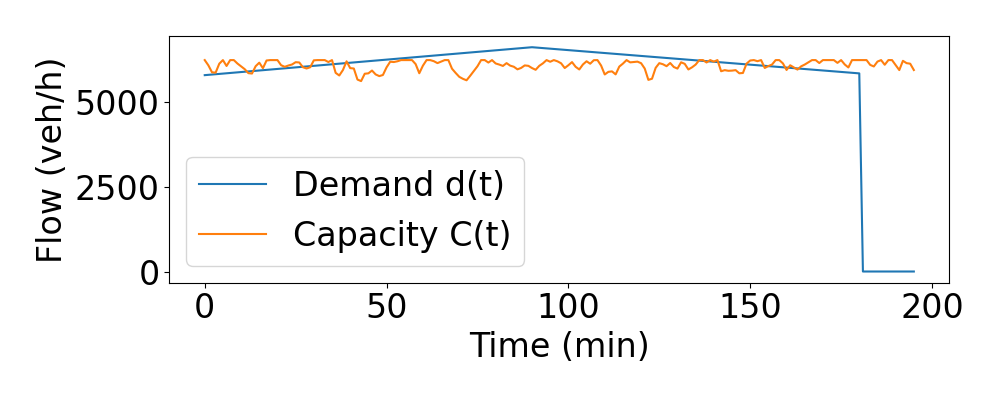}
    \label{fig:sub1}
  \end{subfigure}
  \hfill
  \begin{subfigure}[b]{0.49\textwidth}
    \includegraphics[width=\textwidth]{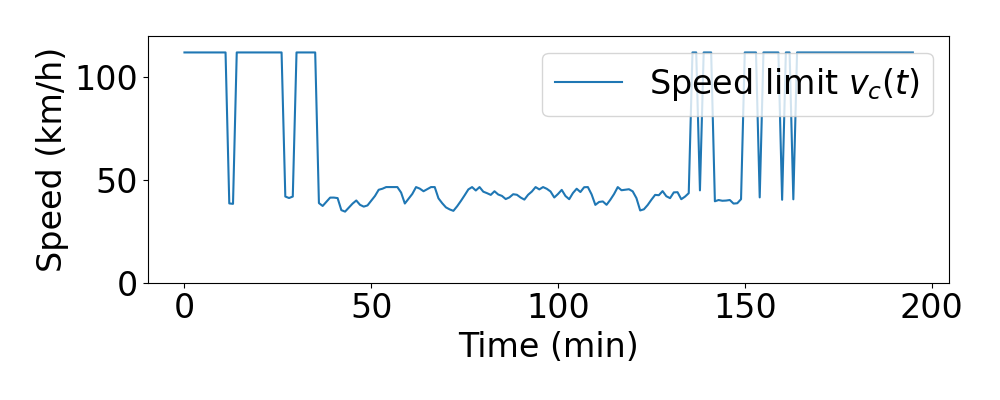}
    \label{fig:sub2}
  \end{subfigure}
  \vskip 1em
  \begin{subfigure}[b]{0.65\textwidth} % make the bottom figure wider
    \includegraphics[width=\textwidth]{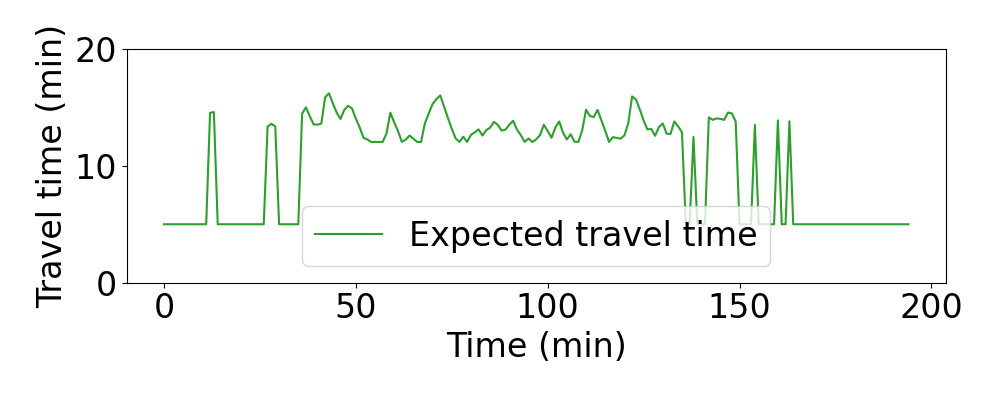}
    \label{fig:sub3}
  \end{subfigure}
  \caption{Stochastic evolution of capacity, VSL, and travel time, $\alpha=0.5$.}
  \label{figure:low_alpha}
\end{figure}

\section{Conclusions}
This study set out to investigate whether VSL control can be tailored to enhance the travel time reliability of commuters  rather than merely to reduce mean travel time. Building on the kinematic‑wave framework and the capacity‑drop mechanism, we formulated two complementary optimization models: (i) a three‑stage stochastic programme in which daily peak inflow is random, and (ii) a stochastic model‑predictive controller that treats the downstream bottleneck capacity as a capped OU process. In each case the objective blended the expected value and the standard deviation of average travel time, converting drivers' risk–sensitive routing preferences into a mathematically tractable control target. Efficient solution procedures - a scalar root‑finding problem for Model 1 and a grid search embedded in a second‑order‑cone MPC for Model 2 - enable real‑time implementation on standard hardware. 

The numerical experiments, calibrated with detector data from a 9km stretch of I‑880N, confirm three key insights. First, when demand is the dominant uncertainty, the optimal strategy exhibits a threshold structure: 
\begin{itemize}
\item When demand is the dominant source of uncertainty, the optimal policy follows a
\emph{travel-time-floor rule}.  Equations \ref{eq:Ti_star} and \ref{eq:crit_alpha} elegantly  reveal that a single threshold
$T^{\star}$ partitions the realized days into two regimes:
(1) Low-demand days ($\bar{T}_{\min} < T^{\star}$), where the controller deliberately raises the average travel time to exactly $T^{\star}$ by imposing a moderate speed reduction. (2) Congested days ($\bar{T}_{\min} \ge T^{\star}$), where the controller achieves the minimum attainable travel time.
The strategy is readily implementable in practice, relying solely on a daily estimate of the unconstrained average travel time and a pre‑computed threshold. The underlying three-stage algorithm is flexible for application in other traffic control settings, as its modular structure enables a clean separation between the control target and the assumptions about traffic flow.
\item When capacity fluctuations prevail, the optimal feedback gain $K^\star$ scales almost linearly with the weight placed on travel‑time variance; greater risk aversion induces earlier, deeper, and longer speed reductions. Third, across both scenarios the proposed VSL policies cut the standard deviation of per‑vehicle travel time by 15–30 \% relative to an unmanaged corridor while limiting the loss in mean travel time to under 5 \%, illustrating the efficiency of explicitly balancing the two moments.
\end{itemize}

These findings have several practical ramifications. For infrastructure owners, they demonstrate that reliability‑oriented VSL can be delivered with modest computational overhead and without installing additional sensing technology beyond conventional loop detectors. For road users, the policies translate into a narrower arrival‑time distribution—an outcome often valued higher than a marginal reduction in mean travel time. From a methodological perspective, the paper shows that reliability objectives do not necessarily complicate optimization: under mild regularity conditions the high\-dimensional control problem collapses to a one‑dimensional threshold search or a low\-dimensional convex programme.

Several limitations should be acknowledged. The analysis assumes (i) a fundamental diagram with homogeneous parameters, (ii) perfect compliance with posted speeds, and (iii) an isolated corridor with a single active bottleneck. Moreover, although the stochastic inflow and capacity processes were estimated from field data, the control algorithm itself was validated only in simulation. 
A natural extension is to co‑optimise the spatial deployment of VSL gantries (the optimization of the location of VSL zones is addresses e.g. in \cite{Martinez2020}). Jointly optimizing (i) the number and location of gantries and (ii) their time‑varying speed settings would convert the problem into a mixed‑integer stochastic programme, raising new questions about tractable relaxations and surrogate models for fast evaluation. Additional work should also explore the integration with ramp‑metering and lane‑use control in a network setting, adaptive estimation of distributional parameters from streaming data and human‑factor studies on how travelers perceive and respond to reliability‑centric speed advice. z

\section*{CRediT authorship contribution statement}

\textbf{Alexander Hammerl:} Conceptualization, Methodology, Formal Analysis, Software, Resources, Investigation, Writing – original draft, Writing – review \& editing, Validation, Visualization.  

\textbf{Ravi Seshadri:} Conzeptualization, Writing – original draft, Writing – review \& editing, Supervision 

\textbf{Thomas Kjær Rasmussen:} Writing – review \& editing, Supervision, Funding Acquisition. 

\textbf{Otto Anker Nielsen:} Writing – review \& editing, Supervision, Funding Acquistion.

\section*{Declaration of competing interest}

The authors declare that they have no known competing financial interests or personal relationships that could have appeared to influence the work reported in this paper.

\section*{Declaration of Use of Generative AI}

During the preparation of this work, the authors used \textbf{Claude 4} and 
\textbf{ChatGPT 5} for stylistic improvements. The authors carefully reviewed the content as needed.

\bibliographystyle{elsarticle-harv}
\bibliography{cas-refs}

\end{document}